\documentclass[11pt]{amsart}
\usepackage[utf8]{inputenc}
\usepackage{amssymb,amsfonts,amsmath,amsthm,latexsym}
\usepackage{graphicx}
\usepackage[active]{srcltx}
\usepackage[english]{babel}

\usepackage{enumerate,verbatim}

\newtheorem{lema}{Lemma}[section]
\newtheorem{theo}[lema]{Theorem}
\newtheorem*{theo*}{Theorem}
\newtheorem{prop}[lema]{Proposition}
\newtheorem{coro}[lema]{Corollary}

\newtheorem{conj}[lema]{Conjecture}
\theoremstyle{remark}
\newtheorem*{rema}{Remark}

\newcounter{teoremaganso}

\def\sideremark#1{\ifvmode\leavevmode\fi\vadjust{\vbox to0pt{\vss 
      \hbox to 0pt{\hskip\hsize\hskip1em           
 \vbox{\hsize2cm\tiny\raggedright\pretolerance10000
 \noindent #1\hfill}\hss}\vbox to8pt{\vfil}\vss}}}%

\title[The trigonometric moment problem ]{On the trigonometric moment problem }
\author{Amelia \'Alvarez, Jos\'{e} Luis Bravo, Colin Christopher}

\begin{document}

\begin{abstract}
The trigonometric moment problem arises from the study of one-parameter
families of centers in polynomial vector fields.
It asks for the classification of the trigonometric polynomials $Q$ which are
orthogonal to all powers of a trigonometric polynomial $P$.

We show that this problem has a simple and natural solution under certain conditions
on the monodromy group of the Laurent polynomial associated to $P$.  In the case of real trigonometric polynomials, which
is the primary motivation of the problem, our conditions are shown to hold for all
trigonometric polynomials of degree $15$ or less.  In the complex case, we show that
there are a small number of exceptional monodromy groups up to degree $30$ where the conditions fail to
hold and show how counter-examples can be constructed in several of these cases.

\end{abstract}
\thanks{
The first author was partially supported by Junta de Extremadura
and FEDER funds. The second author was partially supported by
Junta de Extremadura and a MCYT/FEDER grant number MTM2008-05460.
The first two authors are grateful to the University of Plymouth
for its kind hospitality during the preparation of this work.}

\maketitle

\section{Introduction}
This work is motivated by the following problem proposed by Briskin, 
Fran\c{c}oise and Yomdin~\cite{BFY}:
Given a trigonometric polynomial, $p$, determine all trigonometric polynomials, $q$,
such that the one-parameter Abel equation
\begin{equation}\label{eq:parametricAbel}
z'=p(w)z^2+\epsilon q(w)z^3,
\end{equation}
has a parametric center (i.e., for all $\epsilon$ we have $z(0)=z(2\pi)$ for every trajectory close to $z=0$).
This problem in turn is closely related to the Poincar\'e Center-Focus Problem for planar vector fields, as 
in many cases of interest there are transformations to the trigonometric Abel equation \eqref{eq:parametricAbel} (see \cite{Che}).

In particular,  a necessary and sufficient condition in order for \eqref{eq:parametricAbel} to have a parametric center to first order in $\epsilon$
(see \cite{BFY}) is
\begin{equation}\label{eq:polynomialMoment}
\int_0^{2\pi} P^k(w)dQ(w)=0,\quad k=0,1,2,\ldots
\end{equation}
where $P$ and $Q$ are primitives of $p$ and $q$. The trigonometric moment problem 
consists of given a trigonometric polynomial $P$, obtain all trigonometric polynomials $Q$ such
that \eqref{eq:polynomialMoment} holds. Thus, the moment problem is ``a first order''
approach to the parametric centers. Moreover, Briskin, Roytvarf and Yondim~\cite{BRY} 
have recently proved that the moment problem is indeed equivalent to the parametric center
problem ``at infinity''.

We shall obtain sufficient conditions on $P$ in order to solve the trigonometric moment problem. 
These conditions can be computed for any given trigonometric polynomial $P$ in terms of the
associated monodromy group. Moreover, since the number of possible monodromy groups
for any given degree is finite, they can be computed (theoretically at least) up
to any given degree. We have checked up to degree $15$ and they  
 hold for all trigonometric polynomials with real coefficients. The 
few exceptions we encounter have complex coefficients.

The moment problem has been extensively studied in the polynomial version, thus, where $p$ and $q$ 
are polynomials instead of trigonometric polynomials 
(see, e.g., \cite{BFY1}-\cite{C}, \cite{FPYZ}, \cite{PRitt} ,\cite{Pak2004}, \cite{PRY}). 
Recently it has  been completely solved by Muzychuk and Pakovich~\cite{Pakovich}
who show that $Q$ satisfies \eqref{eq:polynomialMoment}
if and only if there exist polynomials $P_k,Q_k,W_k$, $k=1,\ldots,l$, such that
\[
W_k(0)=W_k(2\pi),\ P(w)=P_k(W_k(w)),\ k=1,\ldots,l,\
\]
\[
Q(w)=\sum_{k=1}^l Q_k(W_k(w)).
\]
The preceding identities can be referred to as a ``weak composition condition''.

Inspired by this result, we study whether a similar result holds for the trigonometric
moment problem.  We fix a trigonometric polynomial $P$ and try to obtain all
trigonometric polynomials $Q$ such that
\begin{equation}\label{eq:trig-moment}
\int_{0}^{2\pi} P^k(\theta)\,dQ(\theta)=0,\quad k=0,1,2,\ldots
\end{equation}

Working instead with respect to a complex variable $z = e^{i\theta}$, this problem is equivalent
to the following: given a Laurent polynomial $P \in \mathcal{L}$, where $\mathcal{L}$ is the space of
complex Laurent polynomials $\mathbb{C}[z,z^{-1}]$, find all $Q \in \mathcal{L}$ such that
\begin{equation}\tag{$A$}\label{eq:moment}
\oint_{|z|=1} P^k(z)\,dQ(z)=0,\quad k=0,1,2,\ldots
\end{equation}

The main difference with the polynomial case is that (up to conjugation by a M\"obius transformation)
the factorization of Laurent polynomials can happen in two distinct ways: a Laurent polynomial $P$ can be written as $\tilde P(z^n)$, where
$\tilde P\in \mathcal{L}$, or as $\tilde P(W)$, where $\tilde P\in \mathbb{C}[z]$
and $W\in\mathcal{L}$ (see \cite{PakovichDec} or \cite{Zieve}).

These two types of factorization give rise to two different mechanisms which imply the
vanishing of the moments \eqref{eq:moment}.  Our aim is to show that these two mechanisms
are enough to solve the trigonometric moment problem for a large number of cases.

Firstly, if $P$ can be written as $\tilde P(z^l)$, $l>1$ with $\tilde P\in \mathcal{L}$,
then it is easy to see (Proposition~\ref{prop:zm}) that if, for some $Q\in \mathcal{L}$, we write $Q(z)=\tilde Q(z^l)+R(z)$
where the coefficients of $z^{nl}$ in $R\in \mathcal{L}$ vanish for all $n\in\mathbb{N}$, then $Q$ satisfies \eqref{eq:moment}
if and only if
\begin{equation}\label{eq:laurent-moment-red}
\oint_{|z|=1} \tilde P^k(z)\,d\tilde Q(z)=0,\quad k=0,1,2,\ldots
\end{equation}
Thus, the problem can be reduced to one of lower degree \eqref{eq:laurent-moment-red}.
We call this condition \eqref{conditB}:
\begin{equation}\tag{$B$}\label{conditB}
\begin{split}
P = \tilde{P}(z^l), \quad Q &= \tilde{Q}(z^l)+\sum_{l\nmid i}a_i z^i,\\
 \quad \oint_{|z|=1} \tilde P^k(z)\,d\tilde Q(z)&=0,\quad k=0,1,2,\ldots
 \end{split}
\end{equation}

On the other hand, if $P$ can be factorized as $P=P_k\circ W_k$ for some $P_k\in \mathbb{C}[z]$ and
$W_k\in\mathcal{L}$, $k=1,\ldots,l$,
then a straight-forward computation (Proposition~\ref{prop:laufactor}) shows that \eqref{eq:moment}
will be satisfied whenever we can write
\begin{equation}
Q=Q_1\circ W_1+\ldots+Q_l\circ W_l,
\end{equation}
for some $Q_k\in \mathbb{C}[z]$.  We call this condition \eqref{conditC}:
\begin{equation}\tag{$C$}\label{conditC}
\begin{split}
P = P_k(W_k)&,\quad P_k, Q_k \in \mathbb{C}[z],\quad W_k \in \mathcal{L}, \quad k = 1,\ldots,l, \\
& Q=Q_1\circ W_1+\ldots+Q_l\circ W_l.
\end{split}
\end{equation}

Our approach to these problems is via the monodromy group of the Laurent polynomial $P$.
In particular we find a condition \eqref{B*} on the monodromy group which is satisfied if and
only if $P =\tilde P(z^l)$ for some $l>1$ (Lemma~\ref{le:blocks})  and hence the problem is reducible to one of lower degree,
and a condition \eqref{C*} which if satisfied implies condition \eqref{conditC} for any $Q$ satisfying \eqref{eq:moment} (Theorem~\ref{theo:method}).

In the following sections we show how the trigonometric moment problem can be translated in terms
of the monodromy group of $P$ and then, using the program GAP, show how all transitive permutation groups of
degree $30$ or less satisfy either condition \eqref{B*} or \eqref{C*} except for a relatively small number of
exceptional groups.  We then study three of the four simplest such groups in more detail (the other
one has been studied by Pakovich, Pech and Zvonkin in the paper~\cite{Pakovichex}).
We obtain corresponding Laurent polynomials, which are of degree 9, 10 and 16, respectively,
and prove the existence of solutions not satisfying conditions \eqref{conditB} or \eqref{conditC}.

In the last section we consider the case of real trigonometric polynomials, which is the most important
case for applications. 
We show that in this case the monodromy group has additional structure and compute
via GAP that all examples up to degree $30$ fall into cases \eqref{conditB} or \eqref{conditC}.  We give
a proof that this is also true whenever $P$ has prime degree, but were unable to prove the general result,
which we leave as a conjecture.

Restating Theorem~\ref{theorem2} in terms of real trigonometric 
polynomials, we obtain the following result.

\begin{theo}

Assume that $P,Q$ are a trigonometric polynomials with real
coefficients up to degree 15. Then \eqref{eq:trig-moment} hold if and only if either 
\begin{enumerate}
 \item There exist a  
polynomial $\tilde P$, and an integer $l>1$, such that 
$P(\theta)=\tilde P(\sin(l\theta),\cos(l \theta))$.
In this case the terms in the Fourier expansion
of $Q$ in $\cos(k\theta),\sin(k\theta)$ with 
$l\nmid k$ make no contribution to \eqref{eq:trig-moment}.
Writting the remaining terms in $Q$ as $\tilde Q(\sin(l\theta),\cos(l\theta))$, 
\begin{equation}\label{eq:condred}
\int_0^{2\pi} \tilde P^i(\sin\theta,\cos\theta)\,d\tilde Q(\sin\theta,\cos\theta)=0. 
\end{equation}
\item There exist polynomials $\tilde P_k$ and $\tilde Q_k$ and trigonometric polynomials $W_k$
such that $P(\theta)=\tilde P_k(W_k(\theta))$ and $Q(\theta)=\sum_k \tilde Q_k(W_k(\theta))$.   
\end{enumerate}
In case (1), condition \eqref{eq:condred} is of strictly smaller degree
than \eqref{eq:trig-moment}. So we can solve the Trigonometric
Moment Problem iterating this theorem.
\end{theo}

%

\section{Preliminaries: Decompositions of $P$ and blocks of $G_P$}


Let us denote by $\mathcal{L}$ the ring of Laurent polynomials
with complex coefficients.
We shall consider $P\in\mathcal{L}$,
\[
P(z)=\sum_{k=-m}^{n} a_k z^k,
\]
where $a_{-m}$ and $a_n$ are non zero and $n,m\geq 1$ (if $n=0$ or $m=0$, then
the problem is solved in~\cite{Prat}). 

A short and neatly simple proof of the next result can be found at page 25 of \cite{Prat}.
For more details see \cite{PakovichDec} or \cite[Lemma~2.1]{Zieve}.

\begin{lema}\label{lema:decompositions}
If $P=g\circ h$, where $P\in\mathcal{L}\backslash(\mathbb{C}[z]\cup \mathbb{C}[z^{-1}])$ and $g,h\in\mathbb{C}(z)$,
then there is a degree-one $\mu\in\mathbb{C}(z)$ such that $G:=g\circ \mu$
and $H:=\mu^{-1}\circ h$ satisfy one of the following cases:
\begin{enumerate}
 \item $G\in\mathbb{C}[z]$ and $H\in\mathcal{L}$.
 \item $G\in \mathcal{L}$ and $H=z^k$ for some $k\in\mathbb{N}$.
\end{enumerate}
\end{lema}

\begin{rema}
The decompositions of the Laurent polynomial $P(z)$ induce
decompositions on the trigonometric polynomial $P(e^{iz})$. Namely,
if $P(z)=\tilde P(z^m)$, with $\tilde P\in\mathcal{L}$, then
$P(e^{iz})$ is a trigonometric polynomial in $\sin(mz)$ and $\cos(mz)$.
If $P(z)=\tilde P(W(z))$, with $\tilde P\in\mathbb{C}[z]$ and $W\in\mathcal{L}$,
then $P(e^{iz})=\tilde P(W(e^{iz}))$, thus, it is the composition of
a polynomial and a trigonometric polynomial.  We show in Section~\ref{tpwrc} that
if we start with a real trigonometric polynomial, the decomposition can also be chosen
to be expressible as the composition of a real polynomial with a real trigonometric polynomial.
\end{rema}

Let $f \in \mathbb{C}(z)$ be a rational function. We say that $z_0 \in \mathbb{C}$ is a {\em critical point} of $f$ if $f'(z_0)=0$, and the value $f(z_0)$ is called a {\em critical value} of $f$. We will denote by $\Sigma$ the set of all critical values of $f$. For $t \in \mathbb{C} \setminus \Sigma$, the set $f^{-1}(t)$ consists of
$m$ different points $z_i(t)$, $i=1,\ldots, m$. By the implicit
function theorem one can push locally each solution $z_i(t)$ to
nearby values of $t$ thus defining multi-valued analytic functions
$z_i(t)$, $t\in\mathbb{C}\setminus\Sigma$. Each loop based at
$t_0 \in \mathbb{C}\setminus\Sigma$ defines thus a permutation of
the $m$ roots $z_1(t_0),\ldots, z_m(t_0)$ of $f(z)=t_0$. These
permutations depend only on the homotopy class of the loop and
form a group $G_f$ called the \emph{monodromy group} of $f$, which is
transitive on the fiber $f^{-1}(t)$ (see, e.g., \cite{Z}). Moreover,
 $G_f$ is the Galois group of the Galois extension of $\mathbb{C}(t)$
by the $m$ pre-images $z_1(t), \ldots, z_m(t)$ of $t \in \mathbb{C} \setminus \Sigma$ by $f$, that is,
\[
G_f = Aut(\mathbb{C}(z_1, \ldots, z_m)/\mathbb{C}(t)).
\]
We recall that an extension $k \to L$ is said to be {\em Galois}
if the degree of the extension coincides with the order of its
automorphism group $G=Aut(L/k) = Aut_{k-alg}(L,L)$:
\[
[ L : k] = | G |.
\]
The automorphism group of a Galois extension is called the {\em
Galois group} of the extension and it is usually denoted by
$Gal(L/k)$, although we will use any of the previous notations
when no confusion can arise.
Any permutation $\sigma \in G_f$ is said to have {\em cycle shape $(d_1)(d_2)\ldots(d_r)$} if $\sigma$ is the product of $r$ disjoint cycles $\sigma_1, \sigma_2, \ldots, \sigma_r$ with respective lengths $d_1 \geq d_2 \geq \ldots \geq d_r$.

Let $P \in \mathcal{L}$ be a Laurent polynomial 
and let $G_P= Gal (L/\mathbb{C}(t))$ be its monodromy group, where $L = \mathbb{C}(z_1(t),\ldots,z_{n+m}(t)))$ and $z_1,\ldots,z_{n+m}$ are the branches of $P^{-1}$. We shall number the branches of $P^{-1}$ such that $\sigma_\infty=(1, 2, \ldots, n)(n+1, n+2, \ldots, n+m)$ is a permutation corresponding to a clockwise loop around infinity.

A useful way to represent the monodromy group $G_P$ of $P$ is the so called ``dessins d'enfants''.
Take $\Sigma=\{t_1,t_2,\ldots,t_r\}$, and let $t_0$ be a non-critical value. Consider
the graph obtained by joining each of the critical values to $t_0$ with non-intersecting paths. Now take the pre-image of that ``star''.
\begin{figure}[h]
\begin{center}
\resizebox{10cm}{!}{\includegraphics{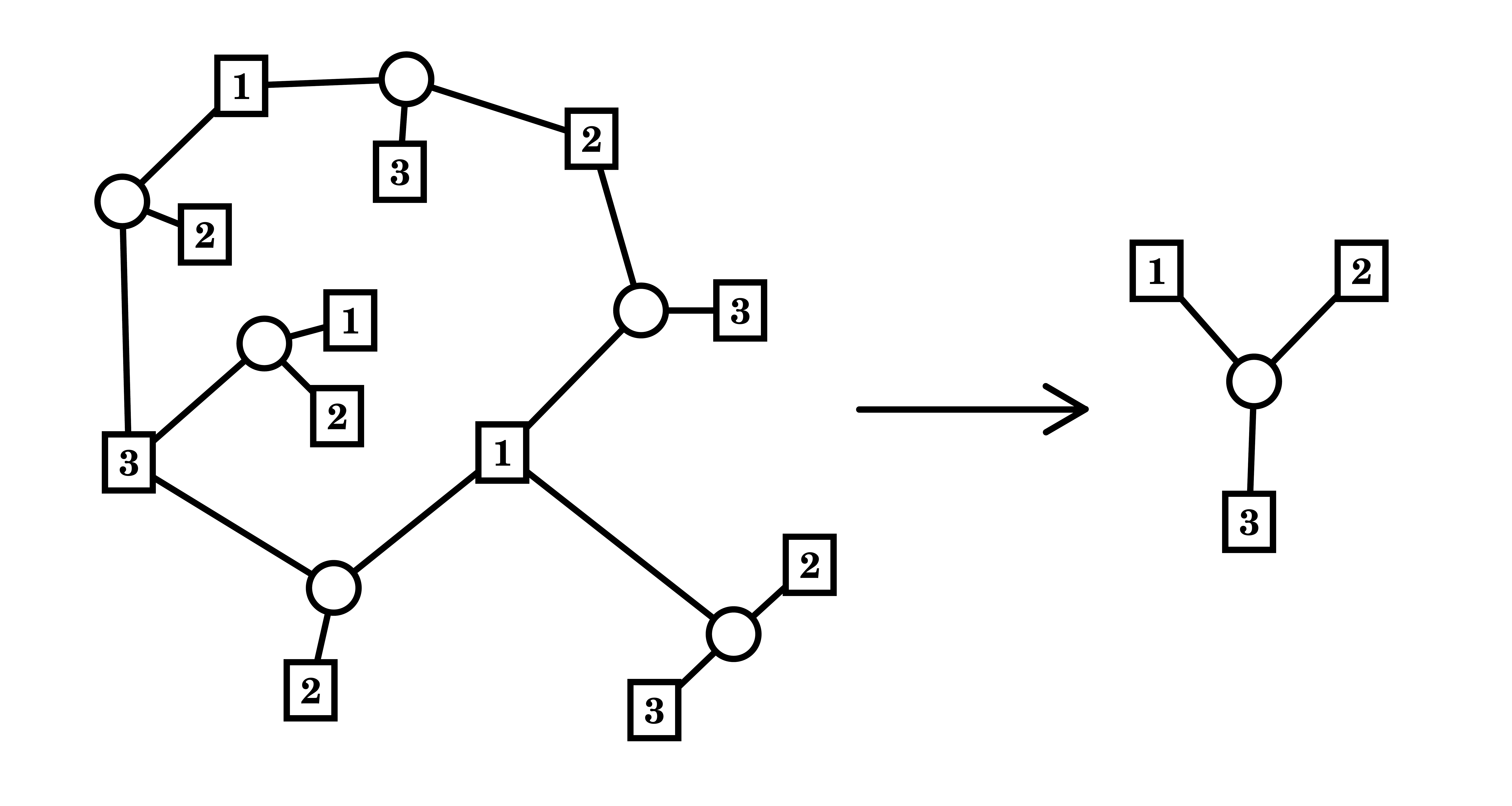}}
\caption{``Dessin d'enfants''} \label{fig:0}
\end{center}
\end{figure}
In Figure~\ref{fig:0}
we have denoted by circles the pre-images of the non-critical value, which can be identified with the branches, and by squares with a number the pre-images of the corresponding critical value. The permutation associated to a simple loop around one of the critical values can be obtained simply by rotating the branches connected with each pre-image of that critical value. This makes clear the relation between a critical value and a permutation associated to a simple loop around it: the critical value has associated critical points $w_1,w_2,\ldots,w_r$ with multiplicity $n_1,n_2,\ldots,n_r$ if and only if the cycle has cycle shape $(n_1)(n_2)\ldots(n_r)$.  In Figure~\ref{fig:0}, the cycles associated to the critical values $1$, $2$ and $3$ have shapes $(3)(2)$, $(2)$ and $(3)$ respectively.

The decompositions of a rational function induce imprimitivity systems in the monodromy group  and vice-versa. Let us describe this in more detail.
A subset $B$ of $X=\{1,2,\ldots,n+m\}$ is called a {\em block}~(\cite{Wielandt}) of $G_P \subseteq S_{n+m}$ if for each $\sigma\in G_P$, either $\sigma(B)=B$ or $\sigma(B)\cap B=\varnothing$. Therefore, for a block $B$, $\mathcal{B}=\{\sigma(B)\colon \sigma\in G_P\}$ is a partition of $X$, which is called an {\em imprimitivity system} of $G_P$. There are always two {\em trivial} imprimitivity systems in $G_P$: the one generated by the block $B=\{1\}$, and the one generated by the block $B=X$.  Let $G_{P,1}$ denote the subgroup of $G_P$ which fixes the element $1$, then the non-trivial block systems of $G_P$ are in one to one correspondence with the subgroups $H$ of $G_P$ strictly lying between $G_{P,1}$ and $G_P$.
Under this correspondence, the subgroups $H$ are exactly the groups which stabilize the block which contains $1$ in the imprimitivity system.

The non-trivial imprimitivity systems for a Laurent polynomial are described in the following result.

\begin{lema}\label{lema:blocks}
Suppose that $n,m\neq 0$, and let $\mathcal{B}$ be a non trivial imprimitivity system of $G_P$.
Then one of the following cases holds:
\begin{enumerate}
\item[(a)] Every block $B\in\mathcal{B}$ is included in or disjoint with $\{1,\ldots,n\}$. In this case, there exists $d|gcd(n,m)$ such that $\mathcal{B}$ coincides with the congruence classes modulo $n/d$ in $\{1,\ldots,n\}$, and with the congruence classes modulo $m/d$ in $\{n+1,\ldots,n+m\}$.

\item[(b)] Every block $B\in\mathcal{B}$ satisfies $B\cap\{1,\ldots ,n\}\neq \varnothing$,
$B\cap\{n+1,\ldots,n+m\}\neq \varnothing$. In this case, there exist $d|gcd(n,m)$, $0\leq r<d$, $k_B \in \mathbb{Z}$ such that
\[
B = \{i \in \{1,\ldots,n\} \colon i \equiv_d k_B \} \cup \{ i \in \{n+1,\ldots,n+m\} \colon i \equiv_d k_B+r \},
\]
that is, such that $B \cap \{1, \ldots, n\}$ is a congruence class modulo $d$ in $\{1, \ldots, n\}$, and $B \cap \{n+1, \ldots, n+m\}$ is a congruence class modulo $d$ in $\{n+1, \ldots, n+m\}$.
\end{enumerate}
\end{lema}

\begin{proof}
(a) Let $B\in\mathcal{B}$ be such that $B\subset\{1,\ldots,n\}$, the other case being analogous. Let $d=|B|$. Then, $\sigma_\infty^i(B)$ is also a block for $1 \leq i \leq n$, and $ \{ \sigma_\infty^i(B) \}_{1 \leq i \leq n}$ is a partition of $\{1,\ldots,n\}$ into subsets of the same cardinality. As a consequence, $\mathcal{B}$ restricted to $\{1,\ldots,n\}$ consists of the congruence classes modulo $n/d$.

Similarly, we can find $\bar B\in\mathcal{B}$ such that $\bar B\subset\{n+1,\ldots,n+m\}$, and repeating the arguments, one obtains that $\mathcal{B}$ restricted to $\{n+1,\ldots,n+m\}$ is a congruence class modulo $m/d$.

(b) If $B \in\mathcal{B}$ verifies that $B\cap\{1,\ldots ,n\}\neq \varnothing$, $B\cap\{n+1,\ldots,n+m\}\neq \varnothing$, then $\{ \sigma_\infty^i(B) \}_{1 \leq i \leq n}$ define one  partition in $\{1, \ldots,n\}$ and another (related) in $\{n+1,\ldots, n+m\}$ of the same cardinality. Let $d=|\{1,\ldots,n\}|/ |B \cap \{1,\ldots,n\}| = | \{ n+1,\ldots, n+m\}|/|B \cap \{n+1,\ldots,n+m\} |$. Then each $B\in\mathcal{B}$ belongs to a congruence class modulo $d$ in $\{1,\ldots,n\}$ and a congruence class modulo $d$ in $\{n+1,\ldots,n+m\}$.
\end{proof}


Given a rational function $f(z)$ of degree $m$ and $\Sigma$ its set of critical values, we say that two {\em decompositions} $ \tilde f \circ h$ and  $ \tilde f' \circ h'$ of $f$ are {\em equivalent} if $h$ and $h'$ define the same field over every pre-image of $t \in \mathbb{C} \setminus \Sigma$ by $f(z)$. That is, if $\mathbb{C}(h(z_i))=\mathbb{C}(h'(z_i))$ for every $i=1,\ldots,m$. There exists a one-to-one correspondence between imprimitivity systems of $G_f$ and equivalence classes of decompositions of $f$ (see \cite{ABM}).  Moreover, for Laurent polynomials we have the following result.

\begin{coro}\label{coro:blocks}
Let $P(z)$ be a Laurent polynomial. There exists a one-to-one correspondence between imprimitivity systems of $G_P$ and decompositions of $P$. More precisely, non-trivial imprimitivity systems are in a one-to-one correspondence with 
Laurent polynomials $W$ such that $P = \tilde P \circ W$ for some $\tilde P \in \mathcal{L}$ up to equivalence by a M\"obius transformation. 
\end{coro}

\begin{proof}
Under the Galois correspondence, $\mathbb{C}(t)$ is the fixed field of $G_P$ and $\mathbb{C}(z_1)$ is the fixed field of the group $G_{P,1}$.  For any imprimitivity system $\mathcal{B}$ of $G_P$, we let $B \in \mathcal{B}$ denote the block containing the element $1$.  The group which fixes this block generates a fixed field $L_{B} \subseteq \mathbb{C}(z_1)$, which by L\"{u}roth's theorem~\cite{W} is generated by a rational function $W \in \mathbb{C}(z_1)$.  Since $\mathbb{C}(t)\subset \mathbb{C}(W)$, we must have $P(z_1) = t = \tilde{P}(W(z_1))$ for some rational function $\tilde{P}$ so that $P = \tilde{P}\circ W$.  Clearly, distinct imprimitivity systems give distinct fixed fields and hence distinct factorizations.
Any other generator of $\mathbb{C}(W(z_1))$ must be of the form $W_1 = \mu \circ W$, for some M\"obius transformation $\mu$.

Conversely, given a non-trivial decomposition $P = \tilde{P}\circ W$ we can form a chain of vector fields $\mathbb{C}(t)\varsubsetneq \mathbb{C}(W(z_1))\varsubsetneq \mathbb{C}(z_1)$, by the Galois correspondence this gives us a subgroup of $G_P$ lying strictly between $G_{P,1}$ and $G_P$, and hence gives rise to an imprimitivity system.
\end{proof}

Moreover, we can determine the type of decomposition.



\begin{lema}\label{le:blocks}
Let $P$ be a Laurent polynomial.
%
The blocks of an imprimitivity system of $G_P$ are contained in or disjoint with $\{1,\ldots,n\}$ if and only if the corresponding $W$ of the decomposition of $P$ is of the form $W(z)=z^k$.
\end{lema}

\begin{proof}

If $P = \tilde{P}(z^k)$, where $\tilde{P}$ is a Laurent polynomial, then we can find roots $\tilde{z}_j(t)$ of $\tilde{P}=t$.
The roots $P(z)=t$ are just the solutions of $z^k = \tilde{z}_j$ for some $j$, and clearly the roots associated to each $\tilde{z}_j(t)$
form an imprimitivity system.  Furthermore, if $\tilde{z}_j(t)$ tends to zero as $t$ tends to infinity then the same is true for all elements of
the block, and similarly if $\tilde{z}_j(t)$ tends to infinity as $t$ tends to infinity.  Hence the blocks are either contained or disjoint from
$\{1,\ldots,n\}$.

Conversely, suppose we have such a block system and let $P(z)=\tilde{P}(W(z))$ denote the corresponding decomposition.  By construction, the blocks of the system correspond to roots of $W(z) = \tilde{z}_j(t)$, where $\tilde{z}_j(t)$ are the roots of $\tilde{P}(z)=t$.  By Lemma~\ref{lema:decompositions}, if we cannot choose $W(z) = z^k$ for some $k$ after composition with a M\"obius transformation, then $\tilde{P}$ must be a polynomial and $W(z)$ a non-trivial Laurent polynomial.  Hence, as $t$ tends to infinity, all the $\tilde{z}_j$ will tend to infinity, and therefore the roots of $W(z)=\tilde{z}_j(t)$ will include both those which tend to zero and those which tend to infinity as $t$ tends to infinity, which contradicts the assumption on the blocks.
\end{proof}

In particular, if $\mathbb{B}$ denotes the set of all imprimitivity systems of $G_P$, then using Lemma~\ref{lema:blocks} we have $P = \tilde{P}(z^l)$, $l>1$, if and only if the following condition holds:
\begin{equation}\tag{$B^*$}\label{B*}
  \exists\, \mathcal{B} \in \mathbb{B} \quad \exists\, B \in \mathcal{B}, \qquad \{1\}\varsubsetneq B \subseteq \{1,\ldots,n\}.
\end{equation}


%


\section{Sufficiency of the weak composition condition}

In this section we study how decompositions of the
Laurent polynomials produce solutions of \eqref{eq:moment}.
Firstly we transform the problem into a more algebraic one.

\begin{prop}\label{prop:equiv}
Let $P,Q\in\mathcal{L}$. The following conditions are equivalent:
\begin{enumerate}
 \item $\oint_{|z|=1} P^k(z)\,dQ(z)=0$ for every $k\geq 0$.
 \item $\oint_{|z|=1} \dfrac{dQ(z)}{P(z)-t}\equiv 0$ for every $t\in\mathbb{C}$.
 \item $\oint_{|z|=1} \dfrac{Q(z)dP(z)}{P(z)-t}\equiv 0$ for every $t\in\mathbb{C}$.
 \item Take $t$ close to infinity and number the pre-images $z_i(t)$ of $P(z)=t$ such that
$\{z_{k}(t)\}_{k=1,\ldots,n}$ (resp. $\{z_{k}(t)\}_{k=n+1,\ldots,n+m}$) are the points close to infinity (resp. zero). Then
\begin{equation}\label{eq:cycle}
\sum_{k=1}^n m Q(z_{k}(t))-\sum_{k=n+1}^{n+m} n Q(z_{k}(t))\equiv 0 \quad\text{ for every }t\in\mathbb{C}.
\end{equation}
\end{enumerate}
\end{prop}

\begin{rema}
In the notation of Gavrilov and Movasati~\cite{GM}, Proposition~\ref{prop:equiv} states that for
any given $P$, then $Q$ is a solution of the Laurent moment problem if and only if
the zero-dimensional abelian integral of $Q$ is identically null along the zero-dimensional cycle
\[
C(t)=\sum_{k=1}^n m Q(z_{k}(t))-\sum_{k=n+1}^{n+m} n Q(z_{k}(t)).
\]
This problem has been recently study for simple cycles ($z_i(t)-z_j(t)$) in \cite{C} and
more generally, for polynomial $P$, in \cite{ABM}.

Note that also in~\cite{Pakovich},~\cite{Pakovichex},~\cite{Prat} the problem is transformed into
the vanishing of a zero-dimensional abelian integral of $Q$ for a certain cycle. This cycle is defined
in terms of the ``dessins d'enfants'', so the explicit computation for
a given polynomial group is not direct. 
\end{rema}

\begin{proof}
Firstly, since $\int_{|z|=1} P^k(z)\,dQ(z)$ are the coefficients of
the Taylor series of \[\oint_{|z|=1} \dfrac{dQ(z)}{P(z)-t}\] at infinity, then
$(1)$ and $(2)$ are equivalent. Moreover, the coefficients of the Taylor series of
\[\oint_{|z|=1} \dfrac{Q(z)dP(z)}{P(z)-t}\]
at infinity are the same except for the first one, which being just $\oint dQ$, is trivially satisfied.
Thus $(3)$ is also equivalent to $(1)$ and $(2)$.

Now, let us prove that $(3)$ and $(4)$ are equivalent.

Take $t$ close to infinity. By the Residue Theorem,
\begin{equation*}
\oint_{|z|=1} \frac{Q(z)dP(z)}{P(z)-t} =2\pi i\left(\sum_{k=n+1}^{n+m} Q(z_{k}(t))-Res(QP'/(t-P),0))\right).
\end{equation*}

Assume that $(3)$ holds. For any $\sigma\in G_P$, by analytic continuation,
\[
\sum_{k=n+1}^{n+m} Q(z_{k}(t))-\sum_{k=n+1}^{n+m} Q(z_{\sigma(k)}(t))=0,
\]
where $\sigma(k)$ just expresses the action of $G_P$ on the roots $z_i$.
If we sum the previous formula over $G_P$, we obtain
\begin{equation}\label{eq:nomonodromy}
\sum_{k=n+1}^{n+m} |G_P| Q(z_{k}(t))-\sum_{k=1}^{n+m}\frac{|G_P|m}{n+m} Q(z_{k}(t))=0.
\end{equation}
Dividing by $|G_P|$ and multiplying by $n+m$ gives
\[
\sum_{k=n+1}^{n+m} n Q(z_{k}(t))-\sum_{k=1}^{n} m Q(z_{k}(t))=0.
\]

Conversely, assume \eqref{eq:cycle} holds (or, equivalently, \eqref{eq:nomonodromy}).
Integrating around a simple cycle containing
no critical point,
\[
\sum_{k=1}^{n+m} Q(z_{k}(t))= Res(QP'/(t-P),0)+ Res(QP'/(t-P),\infty).
\]
Therefore
\[
\begin{split}
\oint_{|z|=1}& \frac{Q(z)dP(z)}{P(z)-t}= 2\pi i\left(\frac{m}{n+m}\sum_{k=1}^{n+m} Q(z_{k}(t))- Res(QP'/(t-P),0))\right)\\
=&2\pi i\left( \frac{m}{n+m} Res(QP'/(t-P),\infty))-\frac{n}{n+m} Res(QP'/(t-P),0))\right).
\end{split}
\]
Since
\[
\begin{split}
\frac{1}{P(z)-t}&=\frac{z^m}{a_{-m}+\sum_{k=1-m}^n a_k z^{k+m}-tz^m}\\
&=\frac{z^m}{a_{-m}}\sum_{l\geq 0}
\left(\frac1{a_{-m}}\sum_{k=1-m}^n tz^m - a_k z^{k+m} \right)^l,
\end{split}
\]
then $Res(QP'/(P-t),0))$ is a polynomial in $t$. Using the change of
variables, $z\to z^{-1}$, $Res(QP'/(P-t),\infty))$ is also a polynomial in $t$.
But as $t$ tends to infinity,
the integral $\oint_{|z|=1} \frac{Q(z)dP(z)}{P(z)-t}$ tends to zero, so it is zero.
\end{proof}

Now we study solutions of \eqref{eq:moment} when $P$ is decomposable.
The first case is when $P$ is a composition with inner term a power of $z$. We
show that in that case the problem can be reduced to one of lower degree.

\begin{prop}\label{prop:zm}
Suppose that $P(z)=\tilde{P}(z^l)$, $l\in\mathbb{N}$, $\tilde{P}\in \mathcal{L}$, and let us write $Q(z)=\tilde{Q}(z^l)+R(z)$, where $R(z)$ contains all the terms of $Q(z)$ which are not divisible by $z^l$. Then $Q$ satisfies  \eqref{eq:moment} if and only if
\begin{equation*}
\oint_{|z|=1} \tilde{P}^n(z)\,d\tilde{Q}(z)=0,\quad n=0,1,2,\ldots
\end{equation*}
\end{prop}
\begin{rema}
Let us recall that $P(z)=\tilde{P}(z^l)$ if and only if
$B\subset\{1,\ldots,n\}$ for some block $B$ of some
imprimitivity system $\mathcal{B}$, thus, this problem can be reduced
to one of lower degree if condition \eqref{B*} holds.
\end{rema}

\begin{proof}
This follows easily by taking into account that
\[
\oint_{|z|=1} \tilde{P}^k(z^l)\,d R(z)=0,\quad k=0,1,\ldots
\]
as the function $\tilde{P}^k(z^l)R'(z)$ can have no poles of order $1$ in the region $|z|\leq 1$.
\end{proof}

The second case is when $P$ can be written as a composition with inner term a
Laurent polynomial. In this case, the weak composition condition is sufficient.

\begin{prop}\label{prop:laufactor}
Suppose that that $P(z)=P_k(W_k(z))$ with $W_k\in\mathcal{L}\backslash(\mathbb{C}[z]\cup \mathbb{C}[z^{-1}])$ for $k = 1,\ldots,l$.
Then $Q(z)=Q_1(W_1(z))+\ldots+Q_l(W_l(z))$ satisfies (\ref{eq:moment}) for all choices of polynomials $Q_k(z)$, $k = 1,\ldots,l$.
\end{prop}

\begin{proof}
By Lemma~\ref{lema:decompositions}, $P_i(z)$ must be a polynomial. Therefore, after the change of variable $u = W_i(z)$,
\[
\oint_{|z|=1} P^k(z)\,d Q_i(W_i)=\oint_{W^{-1}(|z|=1)} P_i^k(z)\,d Q_i(z)=0,\quad k=0,1,2,\ldots
\]
whence the result, since both $P_i$ and $Q_i$ are polynomials.
\end{proof}

\section{Main results}
In this section we obtain a condition \eqref{C*} on the monodromy group $G_P$ which guarantees that every solution of \eqref{eq:moment}
satisfies the weak composition condition \eqref{conditC} if condition \eqref{B*} does not hold.

Take $t_0$ close to infinity and number the branches such that $z_1(t_0)$, $\ldots$, $z_n(t_0)$
are points close to infinity and $z_{n+1}(t_0),\ldots,z_{n+m}(t_0)$ are points close
to zero, and moreover, such that the permutation corresponding to a simple clockwise
loop around infinity (not containing any other singular point) is $(1,2,\ldots,n)(n+1,n+2,\ldots,n+m)$.

Let us define $V$ as the vector space generated by the action of $G_P$ on the coefficients of
\eqref{eq:cycle}, thus,
\[
V=<\sigma(m,\overset{(n)}{\ldots},m,-n,\overset{(m)}{\ldots},-n)\colon \sigma\in G_P >.
\]

Since we are studying when $Q$ satisfies the weak composition condition,
we shall check how to describe compositions in terms of the monodromy group.
Note that if $B$ is a block of an imprimitivity system of the action of $G_P$, and $1\in B$, then
\[
R_B=\sum_{i\in B} Q(z_i(t))
\]
is a function invariant by the action of $G_B$, the stabilizer of $B$ by $G_P$.
Let $L_B$ be the fixed field of $G_B$, then $L_B = \mathbb{C}(W_B(z_1(t)))$ for some $W_B(z)\in\mathbb{C}(t)$ by L\"uroth's theorem.
Hence $R_B=Q_B(W_B(z_1(t)))$ for some $Q_B \in \mathbb{C}(z)$. Moreover, $Q_B\circ W_B\in \mathcal{L}$.  If condition \eqref{B*} does not hold,
$W$ cannot be equivalent to $z^n$, and hence $W_B$ and $Q_B$ can be chosen so that $W_B$ is a Laurent polynomial
and $Q_B$ is a polynomial.  Clearly $W_B$ also appears as a factor in a decomposition of $P = P_B\circ W_B$,
where $P_B$ must also be a polynomial.

For each block $B$ we define $w_B$ to be the vector with $(w_B)_i=1$ when $i\in B$ and $0$ otherwise.
We define $W$ to be the space generated by all vectors $w_B$ where $B$ runs over all blocks which contain the
element $1$ including the trivial block $B = \{1,\ldots,n+m\}$, but not $\{1\}$.

\begin{theo}\label{theo:method}
Let $P$ be a proper Laurent Polynomial ($P\not \in\mathbb{C}[z]$, $P\not \in\mathbb{C}[z^{-1}]$),
such that it does not admit a decomposition of the form $P(z)=\tilde{P}(z^l)$ for any $l>1$.
%
Assume that
\begin{equation}\tag{$C^*$}\label{C*}
(1,0,\ldots,0)\in V+W.
\end{equation}
Then (\ref{eq:moment}) holds if and only if there exist  $P_1,\ldots,P_r,Q_1,\ldots,Q_r\in\mathbb{C}[z]$, $W_1,\ldots,W_r\in\mathcal{L}\backslash(\mathbb{C}[z]\cup \mathbb{C}[z^{-1}])$ such that $P(z)=P_i(W_i(z))$,  $i=1,\ldots,r$, and $Q(z)=Q_1(W_1(z))+\ldots+Q_r(W_r(z))$.
\end{theo}

\begin{proof}

If there exist $P_1,\ldots,P_r,Q_1,\ldots,Q_r\in\mathbb{C}[z]$, $W_1,\ldots,W_r\in\mathcal{L}\backslash(\mathbb{C}[z]\cup \mathbb{C}[z^{-1}])$ such that $P(z)=P_i(W_i(z))$,  $i=1,\ldots,r$, and $Q(z)=Q_1(W_1(z))+\ldots+Q_r(W_r(z))$, then $Q$ satisfies (\ref{eq:moment}) by Proposition~\ref{prop:laufactor}.

Now, suppose that $P$ does not admit a decomposition of the form $P(z)=\tilde{P}(z^l)$,
and that $(1,0,\ldots,0)\in V+W$.
By Proposition~\ref{prop:equiv},
\[
\sum_{k=1}^n mQ(z_k(t))-\sum_{k=n+1}^{n+m} nQ(z_k(t))\equiv 0,
\]
and by analytic continuation and linearity,
\begin{equation}\label{eq:linearQ}
\sum_{k=1}^{n+m} v_i Q(z_i(t))\equiv 0,\quad \text{for every }(v_1,\ldots,v_{n+m})\in V.
\end{equation}

Given a block $B_i$ containing $1$, by the argument above, we have
\[R_{B_i} := \underset{i \in B}{\sum} Q(z_i(t)) = Q_{B_i}(W_{B_i}(z_1)), \]
for some polynomial $Q_{B_i}$ and Laurent polynomial $W_{B_i}$.

If $(1,0,\ldots,0)\in V+W$, then there exists a linear combination
\[
(1,0,\ldots,0) = \sum_i \lambda_i \sigma_i(m,\ldots,m,-n,\ldots,-n) + \sum_j \mu_j w_{B_j}.
\]
Therefore, by (\ref{eq:linearQ})
\[
Q(z_1(t)) = \sum_j \mu_j R_{B_j}(z_1(t)) = \sum_j \mu_j Q_{B_j}(W_{B_j}(z)),
\]
and each $W_{B_j}$ appears as a factor in a decomposition of $P = P_{B_j}\circ W_{B_j}$
where $P_{B_j}$ is a polynomial
\end{proof}

\section{Laurent polynomials up to degree thirty}

Lemma~\ref{le:blocks} and Theorem~\ref{theo:method} give a sufficient condition
on the monodromy group of $P$, such that we can solve the Laurent moment problem.

That is, we can first check if $G_P$ satisfies condition \eqref{B*}.  If so, then the
problem is reducible to a problem of lower degree (and the process repeated).  If not,
we can check whether $G_P$ satisfies condition \eqref{C*}.

Since for every fixed degree $n=\deg(P)$ the monodromy group is a transitive permutation
group of degree $n$, there are only a finite number of possible monodromy groups.  Moreover,
there is a complete list of transitive permutation groups (up to isomorphism) of degree at
most $30$ contained in the computational group theory program GAP \cite{GAP}.

Using GAP, for every transitive permutation group $G$ of degree at most $30$, we consider all possible candidates for
a permutation corresponding to a loop around infinity, that is, we consider all permutations, $g$, which are a product
of two cycles and that move $n$ points (by transitivity, it is only necessary to consider
these cycles up to conjugacy), and all the imprimitivity systems for that group.
Then we can easily check whether condition \eqref{B*} holds or, failing that, whether \eqref{C*} holds.

Moreover, we want to see if they can be realized as monodromy groups of Laurent polynomials.
Given a set of $k$ generators of the group and for each generator a ramification point on the Riemann sphere, we can
invoke Riemann's Existence Theorem to guarantee that this monodromy group is realizable by an algebraic function which is
ramified in exactly these points, and whose monodromy around these points is given by the generators (see \cite{MM}).
Furthermore, we can calculate the genus of the covering given by the algebraic function via the Riemann-Hurwitz formula.
If the genus of the cover is zero for some choice of generators for which the generator around infinity is $g$ as above
then the associated function must be a Laurent polynomial.

The whole process can be automated, and we obtain the following results, where the first row gives the
degree of the permutation group, the second the number of non-isomorphic permutation groups and the final
row the number of exceptions.  For the other possible degrees less than $30$ there are no exceptions (we have not
given the number of possible transitive groups for these degrees).
\begin{center}
\begin{tabular}{crrrrrrrrrr}
$\mbox{Degree}$ & 9 & 10 &  16 & 18 &  20 &   24 & 25 &   27 &  30 \\
$\mbox{Groups}$ &34 & 45 &1954 &983 &1117 &25000 &211 & 2392 &5712 \\
$\mbox{Exceptions}$ & 1 &  2 &   6 &  6 &   3 &   3 &  2 &   31 &   10 \\
\end{tabular}
\end{center}

\begin{theo}  Let $\mathcal{L}_{30}$ be the set
of Laurent polynomials up to degree $30$.

For any $P\in\mathcal{L}_{30}$, if $G_P$ is not one of the exceptional groups in the list above, then
$Q\in \mathcal{L}$ satisfies \eqref{eq:moment} if and only if it is reducible via condition \eqref{conditB} 
to a set of moment equations of lower degree, or satisfies the weak composition condition \eqref{conditC}.  
\end{theo}

Each of the exceptional groups above can be realized as the monodromy group of a Laurent polynomial, although it is only possible to explicitly calculate these polynomials in simple cases.  We have not been able to verify that each exceptional group does indeed give Laurent polynomials $P$ and $Q$ which satisfy \eqref{eq:moment} but neither \eqref{conditB} nor \eqref{conditC}, but we give in detail below, three of the four simplest cases of exceptional groups, showing how they indeed give such $Q$.  A fourth case is considered by Pakovich, Pech and Zvonkin \cite{Pakovichex}.

For the three cases detailed below, we have computed explicitly all Laurent Polynomials $Q$ of bi-degree $(n-1,m)$ satisfying \eqref{eq:moment} but neither \eqref{conditB} or \eqref{conditC}.  To do this we invoke Theorem~7.1 of \cite{Prat} which says that if 
\[\oint_{|z|=1} P^k(z)dQ(z) = 0, \qquad k=0,1,\ldots,(v-1)\deg(Q)+1,\] where $v$ is the number of distinct vectors in the orbit of $(1,\overset{(n)}{\ldots},1,0,\overset{(m)}{\ldots},0)$, then \eqref{eq:moment} holds.

\begin{rema}
For groups of low degree (up to $9$) we have tried to find how many of the transitive permutations groups can be actually realized as monodromy groups of polynomials or Laurent polynomials.  Although our calculations are not complete, at least half of these groups are realizable as such, and so the small number of exceptional cases in the table above does indeed appear to be noteworthy.
\end{rema}

\subsection{Group \boldmath{$A_5(10)$}}
The two groups of degree $10$ are $S_5$ acting on $10$ elements, for which there is an example non-satisfying the weak composition condition~\cite{Pakovichex}, and $A_5$ acting on $10$ elements ($A_5(10)$), with two possible elections of the cycle of infinity (both having cycle shape $(5)(5)$). The two possible infinity cycles for the group $A_5(10)$ are, up to conjugation, $(1,2,3,6,8)(4,9,7,5,10)$ and $(1,2,9,6,7)(3,8,4,10,5)$.

The elements of $A_5(10)$ 
can have cycle shape $(2)(2)(2)(2)$, $(3)(3)(3)$ or $(5)(5)$. Since there must be $10$ critical points, and
previous permutations have $4$, $6$, and $8$ critical points associated, respectively,
then the unique possibility is that there is exactly two critical values, one of cycle shape
$(2)(2)(2)(2)$ and the other $(3)(3)(3)$. Therefore, choosing one of the critical values as zero, and taking the leading coefficient equal to one, one obtains:
\[
P(z)=\frac{(z^3+a_2z^2+a_1z+a_0)^3(z+d)}{z^5},
\]\[
P(z)-\lambda=\frac{(z^4+b_3z^3+b_2z^2+b_1z+b_0)^2(z^2+c_1z+c_0)}{z^5}.
\]
Now deriving,
\[
\begin{split}
P'(z)=&\frac{(z^3+a_2z^2+a_1z+a_0)^2 Q_1(z)}{z^6}
\\&=\frac{(z^4+b_3z^3+b_2z^2+b_1z+b_0) Q_2(z)}{z^6}.
\end{split}
\]
Since $z^3+a_2z^2+a_1z+a_0$ does not divide $z^4+b_3z^3+b_2z^2+b_1z+b_0$,
then $Q_1(z)=z^4+b_3z^3+b_2z^2+b_1z+b_0$. Therefore, we have the following equation
in $a_0,a_1,a_2,\lambda,d,c_0,c_1$:
\[
P(z)-\lambda-\left(\frac{P'(z) z^6}{(z^3+a_2z^2+a_1z+a_0)^2}\right)^2 \frac{z^2+c_1z+c_0}{z^5}=0.
\]
Solving that equation, the solutions with $a_0\neq 0$ are:
\[
P(z)=\dfrac{(z+d) R^3(z)}{z^5},
\]
where
\[
R(z)=\frac{(123 + 55 \sqrt{5})d^3}{2} + \frac{(29 + 13 \sqrt{5})d^2 z}{2} - (2 + \sqrt{5})d z^2 + z^3,
\]
or
\[
R(z)=\frac{(123 - 55 \sqrt{5})d^3}{2} + \frac{(29 - 13 \sqrt{5})d^2 z}{2} - (2 - \sqrt{5})d z^2 + z^3.
\]
We shall take the first solution and $d=1$. 
It
can be proved that its monodromy group is indeed $A_5(10)$,
for instance, computing its ``dessin d'enfants'' (see Figure~\ref{fig:1} where the circles
and squares denote the pre-images of the two critical values). Moreover,
as $A_5(10)$ has no non-trivial imprimitivity systems, $P$ is indecomposable.

\begin{figure}[h]
\begin{center}
\resizebox{7cm}{!}{\includegraphics{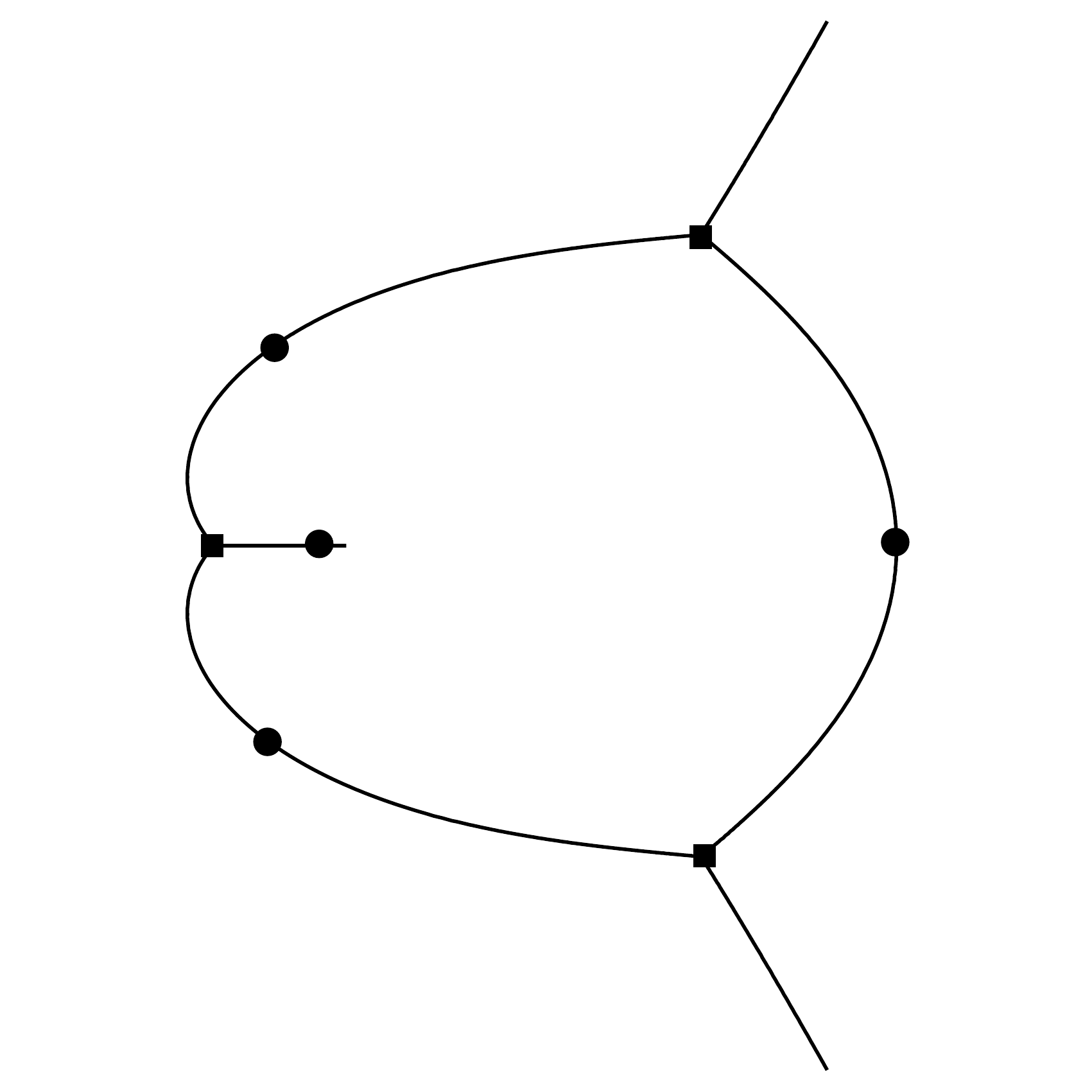}}
\caption{``Dessin d'enfants'' of $P(z)$} \label{fig:1}
\end{center}
\end{figure}

Now, we shall look for the solutions of the form
\[
Q(z)=\sum_{i=-5}^{4} a_i z^i.
\]
Imposing that $\oint_{|z|=1} P^k\,dQ=0$, for $k=0,1,\ldots,5$, one obtains
\[
\begin{split}
a_1 =& \frac{-7 + 3 \sqrt{5}}{2} a_{-1} + 8 \left(-38 + 17 \sqrt{5}\right) a_{-2} +
   3 \left(-928 + 415 \sqrt{5}\right) a_{-3}
 \\&+ 4 \left(445 - 199 \sqrt{5}\right) a_{-4} ,
\end{split}
\]
\[
\begin{split}
a_2 =& \dfrac{-843 + 377 \sqrt{5}}{2} a_{-2} + 3 \left(-3571 + 1597 \sqrt{5}\right) a_{-3}
\\&+ 24 \left(-2889 + 1292 \sqrt{5}\right) a_{-4},
\end{split}
\]
\[
a_3 =  \frac{-15127 + 6765 \sqrt{5}}{2} a_{-3} +
   8\left(-12238 + 5473 \sqrt{5}\right) a_{-4},
\]
\[a_4 = \dfrac{-39603 + 17711 \sqrt{5}}{2} a_{-4},\quad a_{-5}=0.
\]

To prove that $\oint_{|z|=1}P^k\,dQ=0$ for every $k\geq 0$, we use Theorem~7.1 of \cite{Prat}.
This says that it is sufficient to check that $\oint_{|z|=1}P^k\,dQ=0$ for every
$0\leq k\leq (v-1) \deg(Q)+1 = 46$, where we have calculated (via GAP) that $v = 6$.
This can be confirmed very easily via direct computer-aided computation (in our case with 
Mathematica and Maple).

Thus, discounting the constant solution, we have a $4$-dimensional vector space of solutions, $Q$, which do not 
satisfy the weak composition condition.

\subsection{Group \boldmath{$E9:D_8$}}
Now we study the unique group of degree $9$ that we have obtained.
For this group we obtain
that the possible permutation corresponding to a loop
around infinity have cycle shape $(3)(6)$. Moreover, the
permutations of $E9:D_8$ have cycle shape $(2)(2)(2)$, $(2)(2)(2)(2)$, $(3)(3)(3)$, $(4)(4)$, $(3)(6)$.
After trying several combinations, considering two critical values
with permutations associated with cycle shape $(2)(2)(2)$ and $(4)(4)$, and proceeding
as in the group $A_5(10)$, we obtained the following
Laurent polynomial with monodromy group $E9:D_8$ with ``dessin d'enfants'' shown in Figure~\ref{fig:2}
(again we denote the critical points corresponding to one of the critical values with
squares and the corresponding to the other with circles):
\[
P(z)=-\frac{(z-1 )^4 (2 + z) (1 + 2 z)^4}{2 z^3}.
\]
\begin{figure}[h]
\begin{center}
\resizebox{7cm}{!}{\includegraphics[angle=-90]{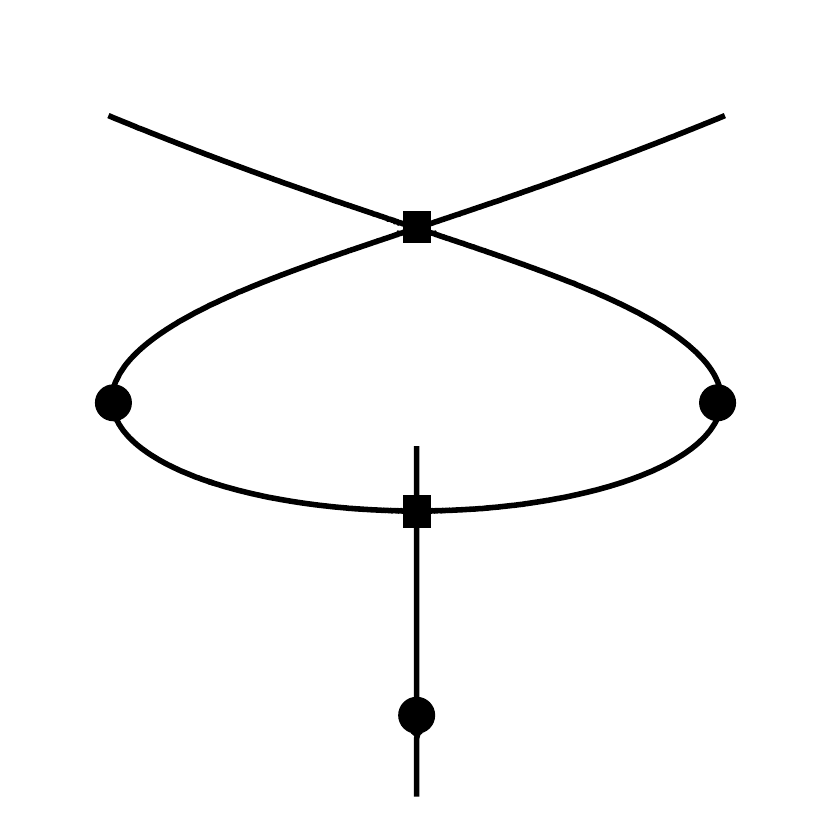}}
\caption{``Dessin d'enfants'' of $P(z)$} \label{fig:2}
\end{center}
\end{figure}
Although it is realizable with Laurent polynomials of degree $9$,
since the permutation corresponding to a loop around infinity has cycle shape $(3)(6)$, the corresponding trigonometric polynomial has degree $12$.

Now we check for the solutions $Q(z)=\sum_{k=-3}^5a_kz^k$, obtaining the three
dimensional space defined by the equations:
\[
a_{-2} = -a_{4}/2,\quad
a_{-1} = -a_{2} - 9 a_{4}/2 + 5 a_{5}/4,\quad
a_{3} = -15 a_{5}/4,\quad a_{-3}=0.
\]
Again, we use Theorem~7.1 of \cite{Prat}, noting that $v$ in this case is $6$,
and check that $\int_{|z|=1}P^k\,dQ=0$ for every $0\leq k\leq (v-1) \deg(Q)+1=41$.

We note that $P$ is indecomposable since the action of $E9:D_8$ has no non-trivial imprimitivity system.
Hence, discounting the constant solution and noting that $a_{1}$ can be chosen arbitrarily in the equations above, 
we obtain a $4$-dimensional space of solutions that do not satisfy the composition condition.

\subsection{Group \boldmath{$t16n195$}}
The simplest case where the action of the monodromy group has a non-trivial imprimitivity system
is the group of degree $16$ with number 195 in GAP. A Laurent polynomial having this group
as its monodromy group (which is represented in Figure~\ref{fig:3}) is
\[
\begin{split}
P(z)=&-940848 + 665280 \sqrt{2} + (89152  - 63040 \sqrt{2}) W(z)
\\&-( 2376  -  1680 \sqrt{2}) W^2(z) + W^4(z),
\end{split}
\]
where,
\[
W(z)=\frac{-99 + 70 \sqrt{2} + (48 - 34 \sqrt{2}) z + (8 - 6 \sqrt{2}) z^3 + z^4}{z^2}.
\]
Since  the group $t16n195$ has a unique non-trivial imprimitivity system, the decomposition above
is the the only one possible.
\begin{figure}[h]
\begin{center}
\resizebox{7cm}{!}{\includegraphics{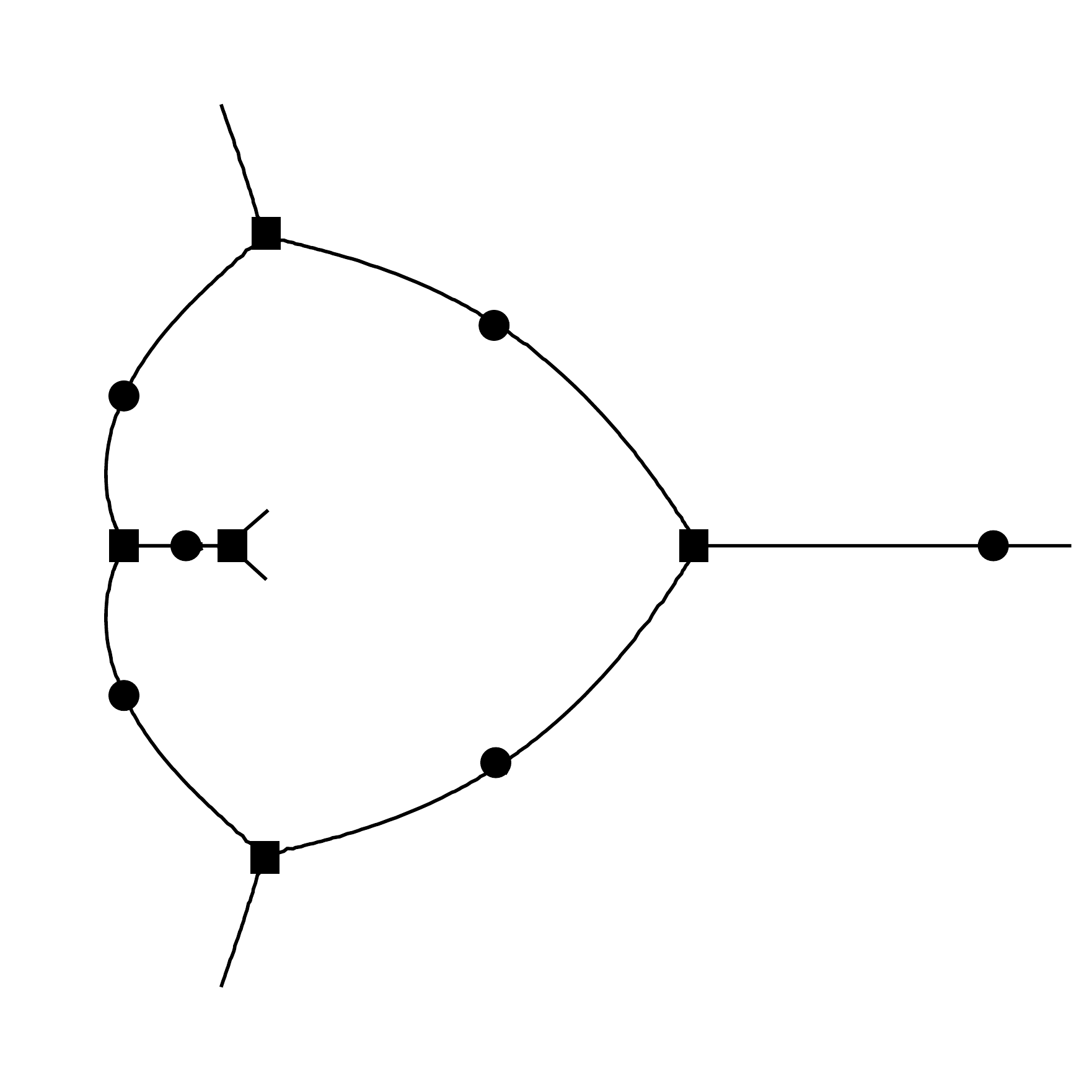}}
\caption{``Dessin d'enfants'' of $P(z)$} \label{fig:3}
\end{center}
\end{figure}

We seek a solution $Q(z)=\sum_{k=-8}^7a_kz^k$.  However, since we can ignore the constant solution, and we already know that
any linear combination of $W$, $W^2$ and $W^3$ are solutions which satisfy the composition conjecture, we shall make the 
assumption that $a_{2k}=0$ for $k=0,1,2,3$.  The remaining solutions must therefore be exceptional ones.

A calculation in GAP shows that $v = 12$ in this case.
Hence, by Theorem~7.1 of \cite{Prat}, it is sufficient to calculate the integral up to $k=155$.  The solutions form 
a 6-dimensional space.  
As the expressions are cumbersome, we only give two generators of this space: 
\[ Q_1 =  1/z^2 + (280 + 198\sqrt{2})z, \qquad Q_2 =  1/z - (17+12 \sqrt{2})z. \]

%
%
%
%


\section{Trigonometric polynomials with real coefficients}\label{tpwrc}

In this section we consider trigonometric polynomials
with real coefficients.  In this case, the Laurent polynomial associated
to the trigonometric polynomial via $z=\exp(i\theta)$, $P$ must satisfy the
equation $\overline{P}(z)=P(1/z)$, where $\overline{P}(z):=\overline{P(\overline{z})}$.
Conversely, any Laurent polynomial
satisfying this equation must correspond to a real trigonometric polynomial.
Every such polynomial must have bi-degree $(n,n)$.

It follows from the above that if $z(t)$ is a solution of $P(z)=t$ for $t\in\mathbb{R}$, $t>>0$, 
then so is $1/\overline{z}(t)$, where $\overline{z}(t) := \overline{z(\overline{t})}$.  Clearly
this map extends to an element of the Galois group $G_P$.  
  
For $t\in\mathbb{R}$, $t>>0$, we label the roots as follows: 
$z_i=k\, w^i\, t^{1/n} + O(1)$ and $z_{n+i} = \overline{k}^{-1}\, w^{1-i}\, t^{-1/n}+O(1)$
where $w$ is a primitive $n$-th root of unity.  The monodromy at infinity corresponds to the
cycle $\sigma_{\infty}:=(1,2,\ldots,n)(n+1,n+2,\ldots,2n)$, and the map $\tau(z_i(t))= 1/\overline{z}_i(t)$ 
is given by the permutation $(1,2n)\cdots(n,n+1)$.  These two elements generate a copy of 
the dihedral group $D_{2n}$.

Adding this extra information into GAP we find that all transitive permutation groups of degree $30$
or less which contain a dihedral subgroup as above must satisfy condition \eqref{B*} or \eqref{C*}.

\begin{theo}\label{theorem2}  Let $\mathcal{L}^*_{30}$ be the set
of Laurent polynomials up to degree $30$ associated to real trigonometric polynomials via $z=\exp(i\theta)$.

For any $P\in\mathcal{L}^*_{30}$, $Q\in \mathcal{L}$ satisfies \eqref{eq:moment} if and only if
either is reducible via condition \eqref{conditB} to a set of moment equations of lower degree, or
satisfies condition \eqref{conditC}.  In the first case, the reduced moment equations still satisfy
the condition $\tilde{P}(z)=\overline{\tilde{P}(1/\overline{z})}$, and so we can iterate this process
of reduction so that all $Q$ satisfying \eqref{eq:moment} can be explained by these two processes.
\end{theo}

\begin{conj}
   For every Laurent polynomial $P$ obtained by a change of variables
from a real trigonometric polynomial, a Laurent polynomial $Q$ satisfies \eqref{eq:moment} if
and only if it satisfies \eqref{conditB} or \eqref{conditC}.
\end{conj}

We note that when condition \eqref{conditB} holds, and $P=\tilde{P}(z^k)$, then
the decomposition corresponds to expressing the original polynomials in terms of 
$\cos(k\theta)$ and $\sin(k\theta)$.  The condition $\overline{P}(z)=P(1/z)$ then
implies that $\tilde{P}$ is a real polynomial.

When condition \eqref{conditC} holds, we want to show that this decomposition can be written
in such a way as to descend to a decomposition of the associated real trigonometric polynomials.

Let $P(z) = \tilde{P}(W(z))$ be the decomposition, where we can choose $\tilde{P}$ to be a 
monic polynomial without loss of generality.  From the condition $\overline{P}(1/z) = P(z)$
we see that the highest and lowest order terms of $W(z)$ must be in the form $k z^r$ and
$\overline{k} z^{-r}$ for some $k\neq 0$ and $r$.  Now, it is easy to see that  
\[\overline{P}(1/z)-P(z) = (\overline{W}(1/z)-W(z))((n/r)k^{n/r-1}z^{n-r} + \cdots ).\]
Since this second term does not vanish, the first must be zero, and hence $W$ is the Laurent
polynomial associated to a real trigonometric polynomial.  It is straight forward to show
that $\overline{P}(z)=P(1/z)$ implies that $\tilde{P}$ must therefore be a real polynomial.

Now we shall prove that the conjecture is true when the degree of $P$ is $2n$, with $n$ a prime number.

\begin{theo}
Suppose that $deg(P)=n>2$ is a prime number. Then $Q\in\mathcal{L}$ satisfies \eqref{eq:moment} if and only if
one of the following possibilities holds:
\begin{enumerate}
\item $P(z)=\tilde{P}(z^n)$, and
\begin{equation*}
\oint_{|z|=1} \tilde{P}^k(z)\,d\tilde{Q}(z)=0,\quad k=0,1,2,\ldots
\end{equation*}
where $\tilde{Q}(z^n)$ is the sum of the monomials of $Q$ divisible by $z^n$.

\item There exist $W\in\mathcal{L}\backslash(\mathbb{C}[z]\cup \mathbb{C}[z^{-1}])$, $\tilde{P},\tilde{Q}\in \mathbb{C}[z]$, such that $P(z)=\tilde{P}(W(z))$, $Q(z)=\tilde{Q}(W(z))$.
\end{enumerate}
Thus, \eqref{eq:moment} hols if and only if \eqref{conditB} or \eqref{conditC} holds.
\end{theo}

\begin{proof}
By Theorem~25.6 of \cite{Wielandt} (or Theorem 13.11.7 of \cite{S}), since $G_P$ contains a dihedral group as a regular subgroup then it is
a $B$-group.  This means that if there are no non-trivial blocks then the group is two-transitive.
In this case $W=<(1,\ldots,1)>$. If we sum up the vector $(1,\overset{(n)}{\ldots},1,-1,\overset{(n)}{\ldots},-1)$ moved by the stabilizer of $1$, we get a vector proportional to $(2n-1,-1,\overset{2n-1}{\ldots},-1)$.
Therefore, it is easy to prove that $(1,0,\ldots,0)\in V+W$
and so we are in Case $(2)$ above with $W = P$. 

Suppose therefore, that a non-trivial block containing $1$ is included in $\{1,2,\ldots,n\}$, then in this case
we can conclude that we are in Case $(1)$ by Theorem~\ref{theo:method}.

Suppose  there is no block containing $1$ included
in $\{1,2,\ldots,n\}$. Then every block containing $1$ is of the form
$\{1,1+n+r\}$ for some $0\leq r<n$.  By the way that $\sigma_\infty$ acts, the rest of the blocks are $\{i,i+n+r\}$, $1\leq i\leq n$.
If there are two such block systems, then taking the numbers $\{1,\ldots,2n\}$
as vertices, and the blocks of the two block systems as edges, we get a graph which decomposes into a number of closed
loops with an even number of edges (the vertices alternate between those in $\{1,\ldots,n\}$ and those in $\{n+1,\ldots,2n\}$,
and the edges alternate between those corresponding to the two block groups). Since the group $G_P$ takes the graph to itself, 
it is easy to see that the intersections of the set of vertices in each of these loops with
$\{1,\ldots,n\}$ and $\{n+1,\ldots,2n\}$ give a block system with the block containing $1$ lying within $\{1,2,\ldots,n\}$.
Hence, we are back in Case $(1)$ again.

Finally, we can suppose that we have just one block system, with blocks $B_i = \{i,i+n+r\}$, $1\leq i\leq n$.
Since $\{1,\ldots,n\}$ is not part of a block system, then there is at least one $\sigma\in G_P$ for which $\{1,\ldots,n\}$
has a non-empty image in both $\{1,\ldots,2n\}$ and $\{n+1,\ldots,2n\}$.  

Let $v_i$ be the vector with $1$ in the $i$-th place and $-1$ in the $i+n+r$-th place.  
The element $\sigma_\infty$ of $G_P$ acts on the $v_i$ as a cyclic group of order $n$.
Since $n$ is prime, there are only two rational invariant subspaces of the space generated
by the $v_i$ under the action of $\sigma_\infty$.  The one is generated by $\sum v_i$ and the other 
is generated by all $\sum c_i v_i$ with $\sum c_i = 0$.  Any $\sum d_i v_i$ with $\sum d_i \neq 0$
and with not all $d_i$ the same, must therefore generate the whole space under the action of $\sigma_\infty$.

Let $v=(1,\overset{(n)}{\ldots},1,-1,\overset{(n)}{\ldots},-1) = \sum v_i$, then since the action of $\sigma$ respects the blocks, 
we see that $(\sigma(v)-v)/2 = \sum d_i v_i \in V$, where $d_i \in \{0,1\}$.  By our choice of $\sigma$, not all the $d_i$ are zero nor
are they all $1$ and hence the element $v_1$ must also lie in $V$.  Finally, let $v_{B_1}\in W$ be the vector associated to the block $B_1$,
then we get $(1,0,\ldots,0)=(v_1+v_{B_1})/2\in V+W$.  We are therefore in Case $(2)$.
\end{proof}


\end{document}